\documentclass[11pt, reqno]{amsart}

\usepackage[dvipsnames,usenames]{color}
\usepackage[colorlinks=true, urlcolor=NavyBlue, linkcolor=NavyBlue, citecolor=NavyBlue]{hyperref}
\usepackage{float}
\usepackage{amsmath, amsthm, amssymb}
\usepackage{amsfonts}
\usepackage{enumerate}
\usepackage{epsf}
\usepackage{psfrag}
\usepackage{amsmath}
\usepackage[all]{xy}
\usepackage[dvips]{graphicx}
\usepackage{epsfig}
\usepackage{epstopdf}
\usepackage{ifpdf}
\usepackage[UKenglish]{babel}
\usepackage{subfigure}
\usepackage{hyperref}
\usepackage{color}
\usepackage{marginnote}

\hypersetup{
linkbordercolor={1 0 0}, % set to white
citebordercolor={0 1 0} % set to white
}
\newtheorem{thm}{Theorem}[section]

\newtheorem{prop}[thm]{Proposition}

\newtheorem{examp}[thm]{Example}

 \theoremstyle{definition}
\newtheorem{defn}[thm]{Definition}
\theoremstyle{remark}
\newtheorem{rmk}[thm]{Remark}

\numberwithin{equation}{subsection}
\numberwithin{enumi}{subsection}

\makeatletter

\newcommand{\Rmnum}[1]{\expandafter\@slowromancap\romannumeral #1@}
\makeatother
\author[Andrew Donald]{Andrew Donald}
\thanks{}
\address {Department of Mathematics, Michigan State University, East Lansing, MI 48824}
\email{adonald@math.msu.edu}

\author[Faramarz Vafaee]{Faramarz Vafaee}
\thanks{}
\address {Mathematics Department, California Institute of Technology, Pasadena, CA 91125}
\email{vafaee@caltech.edu}

\begin{document}
\title{A slicing obstruction from the $\frac{10}{8}$ theorem}

\begin{abstract}
From Furuta's $\frac{10}{8}$ theorem, we derive a smooth slicing obstruction for knots in $S^3$ using a spin $4$-manifold whose boundary is $0$-surgery on a knot. We show that this obstruction is able to detect torsion elements in the smooth concordance group and find topologically slice knots which are not smoothly slice.

\end{abstract}
\maketitle

\section{Introduction}\label{section:1}

A knot $K$ in $S^3$ is \emph{smoothly slice} if it bounds a disk that is smoothly embedded in the four-ball. Although detecting whether or not a knot is slice is not typically an easy task to do, there are various known ways to obstruct sliceness. For instance, the Alexander polynomial of a slice knot factors, up to a unit, as $f(t)f(t^{-1})$ and the averaged signature function of the knot vanishes (see, for instance, \cite[Chapter~8]{Lickorish1997}). Also in recent years, modern techniques in low-dimensional topology have been applied to produce obstructions. Examples include the $\tau$-invariant \cite{Ozsvath2003, Rasmussen2003}, $\epsilon$ \cite{Hom2014} and $\Upsilon$ \cite{Ozsvath2014} invariants, all coming from Heegaard Floer homology \cite{Ozsvath2004a, Ozsvath2013}, and the $s$-invariant \cite{Rasmussen2010} from Khovanov homology \cite{Khovanov2000}. In this paper we introduce a new obstruction using techniques in handlebody theory. We call a $4$-manifold a \emph{2-handlebody} if it may be obtained by attaching $2$-handles to $D^4$. The main ingredient is the following: 
{\thm\label{thm:slicing}%this is based on a result from donald2012 but does not appear stated in the form we need
Let $K \subset S^3$ be a smoothly slice knot and $X$ be a spin 2-handlebody with $\partial X = S^3_0(K)$. Then either $b_2(X)=1$ or 
\[
4b_2(X) \ge 5|\sigma(X)|+12.
\]}

A key tool used in the proof of Theorem~\ref{thm:slicing} is Furuta's $10/8$ theorem \cite{Furuta2001}. Our theorem can be regarded as an analogous version of his theorem for manifolds with certain types of boundary. Similar ideas to this paper have been used by Bohr and Lee in \cite{Lee2001}, using the branched double cover of a knot.

Given a knot $K$, we construct a spin 4-manifold $X$ such that $\partial X = S^3_0(K)$. If we think of $0$-surgery on $K$ as the boundary of the manifold given by a single 2-handle attached to $\partial D^4$, the spin structures on $S^3_0(K)$ are in one-to-one correspondence with characteristic sublinks in this Kirby diagram. (See Section~\ref{section:2} for the relevant definitions.) The 0-framed knot $K$ represents a spin structure which does not extend over this $4$-manifold. We may alter the $4$-manifold, without changing the boundary $3$-manifold, by a sequence of blow ups, blow downs and handle slides, until the characteristic link corresponding to this spin structure is the empty sublink. The manifold we obtain is a spin 4-manifold. Now if $b_2$ and $\sigma$ of the resulting four-manifold violate the inequality of Theorem~\ref{thm:slicing}, $K$ is not smoothly slice.

The reason we are interested in the obstruction obtained from Theorem~\ref{thm:slicing} is twofold. First, we show in Section~\ref{sec:4} that our obstruction is able to detect torsion elements in the concordance group; in particular, the obstruction detects the non-sliceness of the figure eight knot. Second, we show that the obstruction is capable of detecting the smooth non-sliceness of topologically slice knots. We remind the reader that a topologically slice knot is a knot in $S^3$ which bounds a locally flat disk in $D^4$. All the algebraic concordance invariants (e.g. the signature function) vanish for a topologically slice knot. 

\section*{Acknowledgements} We would like to thank Matthew Hedden, Adam Levine, and Yi Ni for insightful conversations and their interest in our work. We would also like to thank Jae Choon Cha, David Krcatovich, and Brendan Owens for helpful discussions and their comments on an earlier draft of this paper. Lastly, we thank the referee for useful comments.
%I don't think we want to say this. It may be appropriate to thank the referee for useful or helpful comments? But favourable sounds off.
\section{The Slicing Obstruction}\label{section:2}
In this section we prove Theorem \ref{thm:slicing} and describe how to produce the spin manifolds used to give slicing obstructions. The argument uses Furuta's $10/8$ Theorem.

{\thm\cite[Theorem~1]{Furuta2001}
Let $W$ be a closed, spin, smooth 4-manifold with an indefinite intersection form. Then \[4b_2(W) \geq 5 |\sigma(W)| +8.\]}

Note that, by Donaldson's diagonalisation theorem \cite{Donaldson1987}, a closed, smooth, spin manifold $W$ can have a definite intersection form only if $b_2(W)=0$.% To apply this, we need to construct closed, spin 4-manifolds.

\begin{proof}[Proof of Theorem \ref{thm:slicing}]
We start by noting that when $K$ is smoothly slice, $S^3_0(K)$ smoothly embeds in $S^4$. (See \cite{Gilmer-Livingston}, for example.) The embedding splits $S^4$ into two spin manifolds $U$ and $V$ with common boundary $S^3_0(K)$. Since $S^3_0(K)$ has the same integral homology as $S^1 \times S^2$, a straightforward argument using the Mayer-Vietoris sequence shows manifolds $U$ and $V$ will have the same homology as $S^2 \times D^2$ and $S^1 \times D^3$ respectively. In particular both spin structures on the three-manifold extend over $V$.

Now, as in \cite[Lemma~5.6]{Donald2015}, if $X$ is a spin 2-handlebody with $\partial X = \partial V$, let $W=X \cup_{S^3_0(K)} -V$. This will be spin and $\sigma(W) = \sigma(X)$ since $\sigma(V)=0$. In addition, we have $\chi(W)=\chi(X)=1+b_2(X)$. Since $H_1(W,X;\mathbb{Q}) \cong H_1(V,Y;\mathbb{Q})=0$ it follows from the exact sequence for the pair $(W,X)$ that $b_1(W)=b_3(W)=0$. Therefore $b_2(W) = b_2(X) -1$. The result follows by applying Furuta's theorem in the case $b_2(X) >1$.
 \end{proof}

The rest of this section provides the background needed to apply the obstruction of Theorem~\ref{thm:slicing}. We refer the reader to \cite{Gompf1999} for a more detailed discussion on spin manifolds and characteristic links.
%Throughout, we will be assuming that $Y$ is a spin 3-manifold obtained from $0$-surgery on a knot $K \subset S^3$.

{\defn \label{spin} A manifold $X$ has a spin structure if its stable tangent bundle $TX\oplus \epsilon^k$, where $\epsilon^k$ denotes a trivial bundle, admits a trivialization over the 1-skeleton of $X$ which extends over the 2-skeleton. A spin structure is a homotopy class of such
trivializations.
}
\\

It can be shown that the definition does not depend on $k$ for $k \ge 1$. An oriented manifold $X$ admits a spin structure if the second Stiefel-Whitney class vanishes, that is $\omega_2(X)=0$. An oriented 3-manifold always admits a spin structure, since its tangent bundle is trivial. We remind the reader that any closed, connected, spin $3$-manifold $(Y, \mathfrak{s})$ is the spin boundary of a $4$-dimensional spin $2$-handlebody. A constructive proof is given in \cite{Kaplan1979}. %One way to prove this well-known fact starts by noting that every orientable $3$-manifold $Y$ bounds a $2$-handlebody. Moreover, every spin $3$-manifold $Y$ is spin cobordant to $S^1 \times \Sigma_g$ for some $g$ (and some spin structure on the latter manifold). Here, $\Sigma_g$ denotes the Riemann surface of genus $g$. The $3$-manifold $S^1 \times \Sigma_g$ (with any spin structure) is spin cobordant to the disjoint union of several $3$-tori. Finally, the $3$-dimensional torus with any of its spin structures is a spin boundary. See~\cite{Stipsicz2000} for more details. See also~\cite{Marin1986, Kirby1989}. 

As described in Section~\ref{section:1} we are interested in $0$-surgery on knots. The resulting three-manifold is spin with two spin structures $\mathfrak{s}_0, \mathfrak{s}_1$. Note that one of the spin structures, $\mathfrak{s}_0$, extends to the 4-manifold obtained by attaching a $0$-framed $2$-handle to $D^4$ along the knot. There is another $2$-handlebody (not the one with one $2$-handle that $\mathfrak{s}_0$ extends over) that is also bounded by $S^3_0(K)$ and $\mathfrak{s}_1$ extends over it. We explain how to construct such a four-manifold in what follows. 
{\defn\label{charactersitic} Let $L=\{K_1, ..., K_m\}$ be a framed, oriented link in $S^3$. The linking number $lk(K_i, K_j)$ is defined as the linking number of the two components if $i \neq j$ and is the framing on $K_i$ if $i=j$. A characteristic link $L^{'}\subset L$ is a sublink such that for each $K_i$ in $L$, $lk(K_i, K_i)$ is congruent mod 2 to the total linking number $lk(K_i, L^{'})$.}    
\\

Note that the characteristic links are independent of the choice of orientation of $L$. A framed link is a Kirby diagram for a $2$-handlebody $X$ and the characteristic links are in one-to-one correspondence with spin structures on $\partial X$. The link components form a natural basis for $H_2(X)$ and the intersection form is given by the linking numbers $lk$. The empty link is characteristic if and only if this form is even and, since $2$-handlebodies are simply connected, this occurs if and only if $X$ is spin. A non-empty characteristic link correspond to a spin structure on the boundary which does not extend. We can remove a characteristic link by modifying the Kirby diagram by handle-slide, blow up and blow down moves until it becomes the empty sublink. These do not change the boundary $3$-manifold, but the latter two change the $4$-manifold. This process produces a spin $4$-manifold where the given spin structure extends.

For convenience, we briefly recall how these moves change the framings in link and the effect on a characteristic link. When a component $K_1$ with framing $n_1$ is slid over $K_2$ with framing $n_2$, the new component will be a band sum of $K_1$ and a parallel copy of $K_2$. It will have framing $n_1 + n_2 + 2 lk(K_1,K_2)$, where this linking number is computed using orientations on $K_1$ and $K_2$ induced by the band. The new component will represent the class of $K_1 + K_2$ in $H_2(X)$. Consequently, if $K_1$ and $K_2$ were part of a characteristic link before the slide, the new component will replace them in the new diagram. The most basic blow up move adds a split unknot with framing $\pm 1$. Each characteristic link will change simply by adding this extra component. A general blow up across $r$ parallel strands consists of first adding a split component and then sliding each of the $r$ strands over it. Therefore blowing up positively (respectively negatively), if the linking of the blow up circle with a component of the Kirby diagram is $p$, the framing change on that component will be $p^2$ (respectively $-p^2$). If a blow up curve links a characteristic link non-trivially mod 2 then it does not add any components to the characteristic link. However, if the blow up curve circles $2k$  strands of a characteristic link, it will be added to the characteristic link. Example~\ref{eg:figureeight} (more specifically, Figure~\ref{fig8b} and Figure~\ref{fig8c}) illustrates this. A blow down is the reverse move. Blowing down a component of a characteristic link removes it.

Note that during the process of removing a characteristic link, we do not need to keep track of the whole Kirby diagram. Instead, we need only keep the information about the characteristic link and its framings, along with $b_2$ and $\sigma$. This is straightforward to do by counting the number of blow ups and blow downs with their signs.
\subsection{Obtaining a spin 4-manifold bounded by $S^3_0(K)$}

The argument above suggests that Theorem~\ref{thm:slicing} can give slicing obstructions for a knot $K$ that can be ``efficiently'' unknotted by a sequence of blow-ups. If the characteristic link is an unknot, the framing can be transformed to $\pm 1$ by further blow ups (along meridians) and then we may blow down to get an empty characteristic link.

We finish this section by showing how Theorem~\ref{thm:slicing} can be used to prove that positive $(p, kp\pm1)$ torus knots are not smoothly slice for odd $p \ge 3$ \footnote{There are many ways to show that positive torus knots are not smoothly slice. Our goal in presenting this example is to show that our obstruction works well with \emph{generalized twisted torus knots}, which are, roughly speaking, torus knots where there are full-twists between adjacent strands. See Figure~\ref{fig:k6} for an example of a generalized twisted torus knot.} . Given a zero framed positive $(p, kp\pm1)$ torus knot, we first blow up $k$ times negatively around $p$ parallel strands. Each will introduce a negative full twist and, since $p$ is odd, the characteristic link will be a $-kp^2$ framed unknot. Blowing up $kp^2-1$ times positively along meridians and blowing down once negatively will give us a spin manifold $X$. This sequence used $k$ negative blow ups, $kp^2-1$ positive blow ups and one negative blow down so we see $b_2(X)= 1+k+kp^2-1-1= kp^2+k-1$ and  $\sigma(X) = -k+kp^2-1+1=kp^2-k$. Now, $4b_2(X)-5|\sigma(X)|-12= -kp^2+9k-16<0$, and so such knots are not slice.

\section{Examples}\label{sec:4}

The obstruction from Theorem \ref{thm:slicing} is able to detect knots with order two in the smooth concordance group and can also be used to obstruct topologically slice knots from being smoothly slice. This section describes examples which illustrate each of these properties.

\subsection{Figure eight knot}

\begin{examp}\label{eg:figureeight} The knot $4_1$ is not slice.
\end{examp}

This knot is shown in Figure \ref{fig8a}. Start with the manifold obtained by attaching a $0$-framed $2$-handle to $D^4$ along $4_1$. Blow up the manifold twice as indicated in Figure \ref{fig8b}. Sliding one of the two blow up curves over the other results in the diagram in Figure~\ref{fig8c}. The characteristic link is a split link whose components are a $0$-framed trefoil and a $-2$-framed unknot.
%Note: Choose to blow up in the way that is convenient for k6. It isn't relevant in this example at all

Figure \ref{fig8d} shows just the characteristic link. Blowing up negatively once more changes the characteristic link to a $2$-component unlink with framings $-2$ and $-9$ as in Figure \ref{fig8e}. This is inside a $4$-manifold with signature $-3$ and second Betti number $4$. Positively blowing up meridians nine times changes both framings in the characteristic link to $-1$ and blowing down each of them results in a spin manifold. Counting blow-up and blow-down moves, we see that the signature of this spin manifold is $+8$ and the second Betti number is $11$. Theorem \ref{thm:slicing} then applies.

\begin{figure}[t!] 
\label{fig8}
\subfigure{\label{fig8a}
\def\svgwidth{3cm}
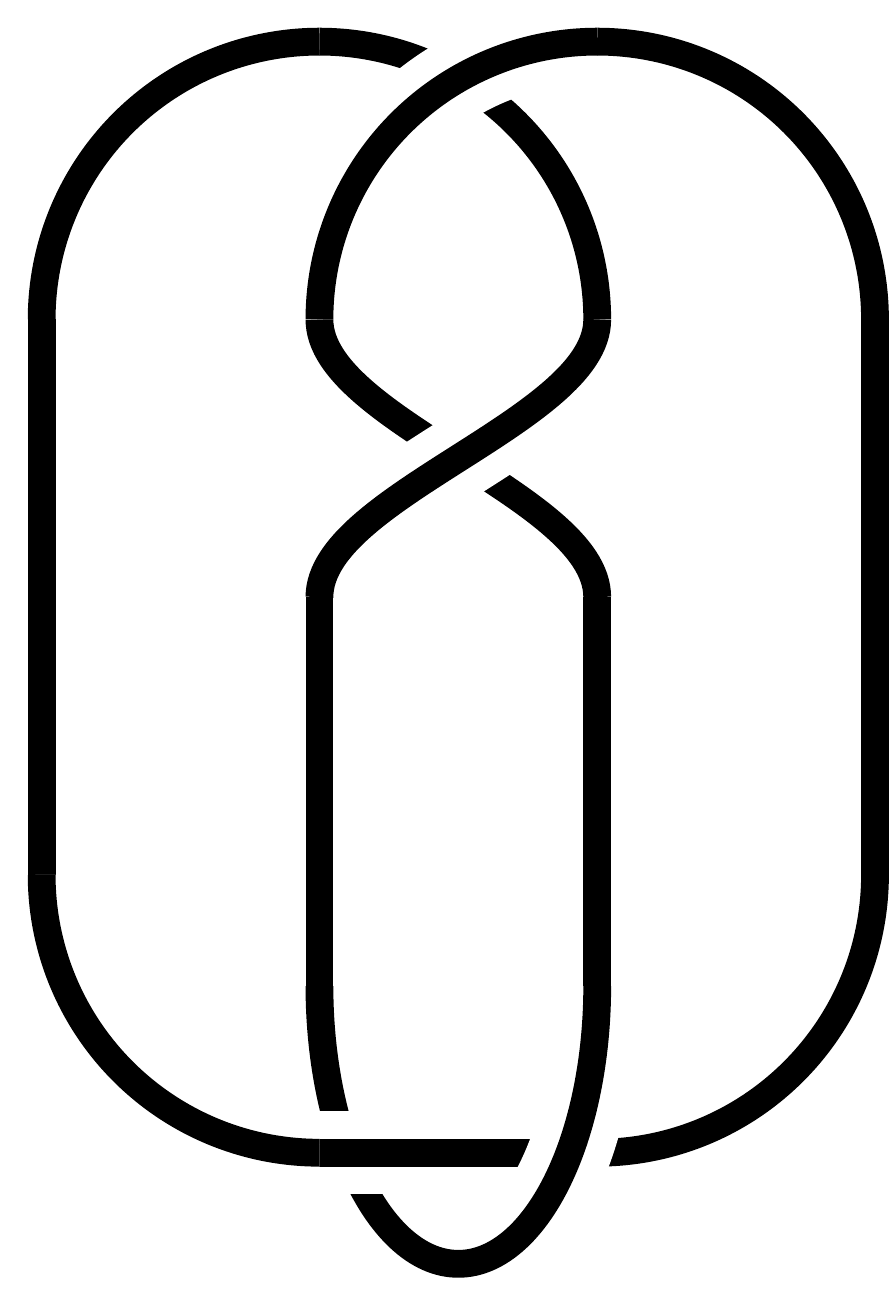{(a)}}
~~~~~
\subfigure{\label{fig8b}
\def\svgwidth{4cm}
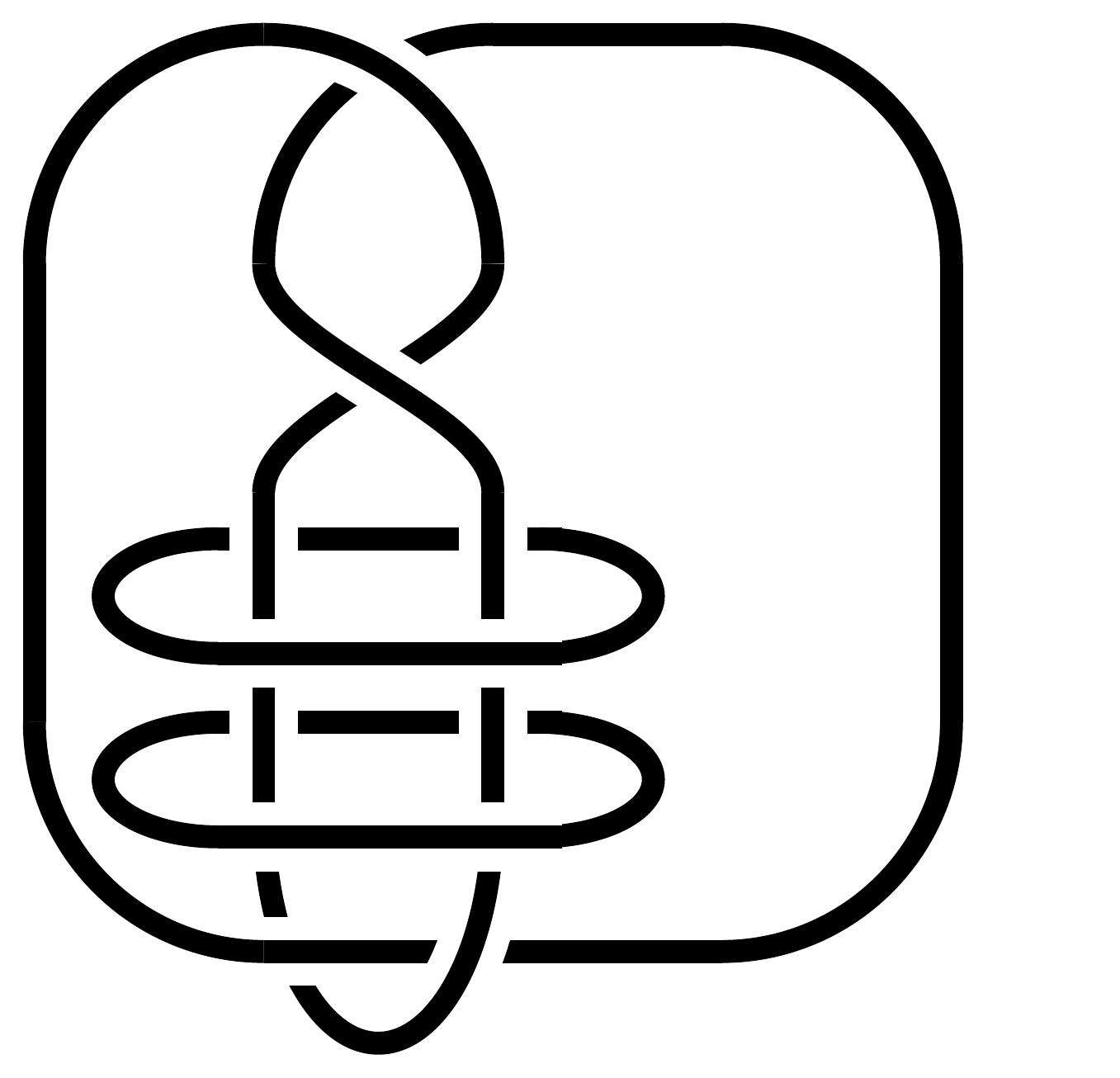{(b)}}
~~~
\subfigure{\label{fig8c}
\def\svgwidth{4cm}
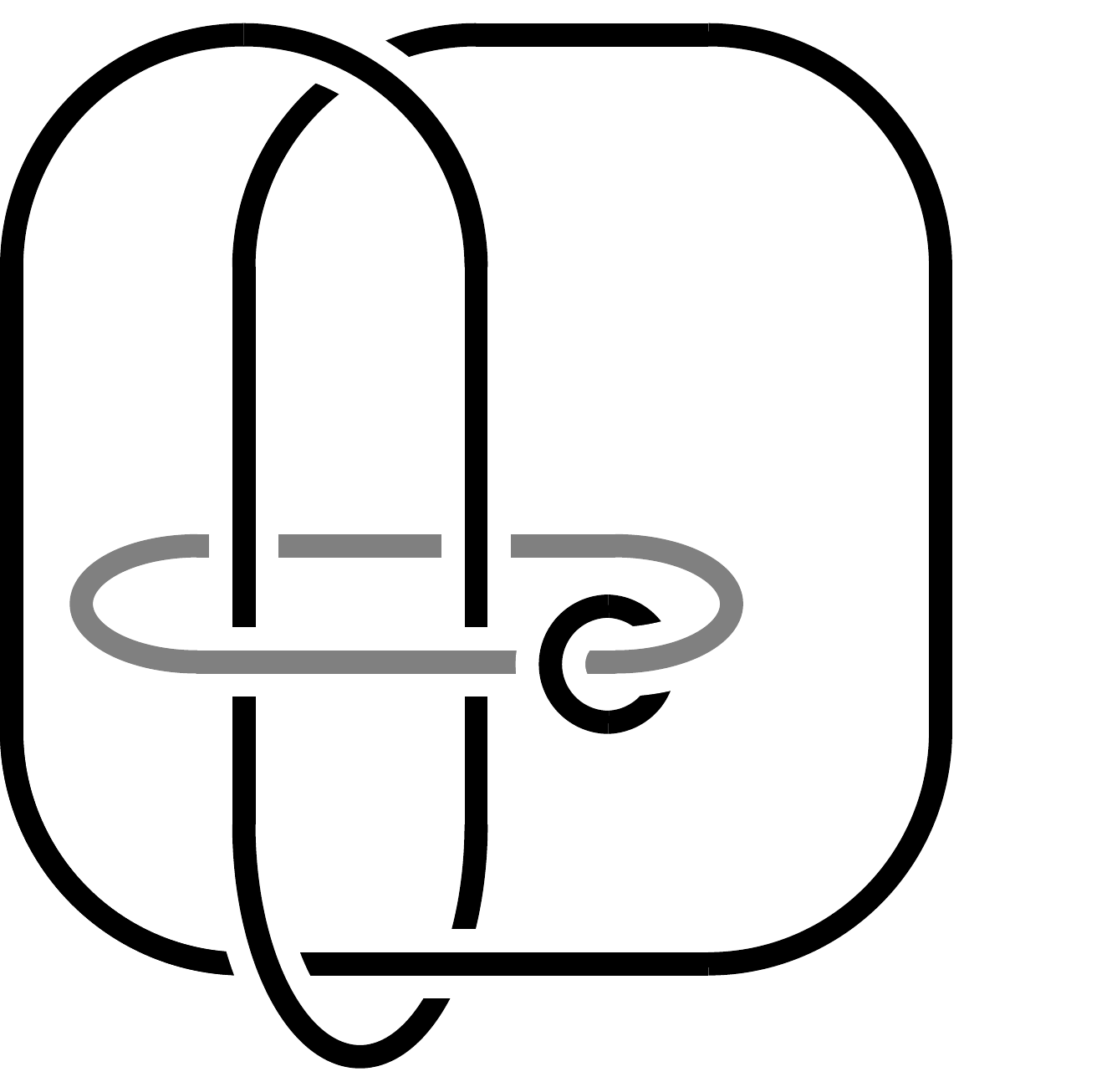{(c)}}
\par\bigskip
\subfigure{\label{fig8d}
\def\svgwidth{5cm}
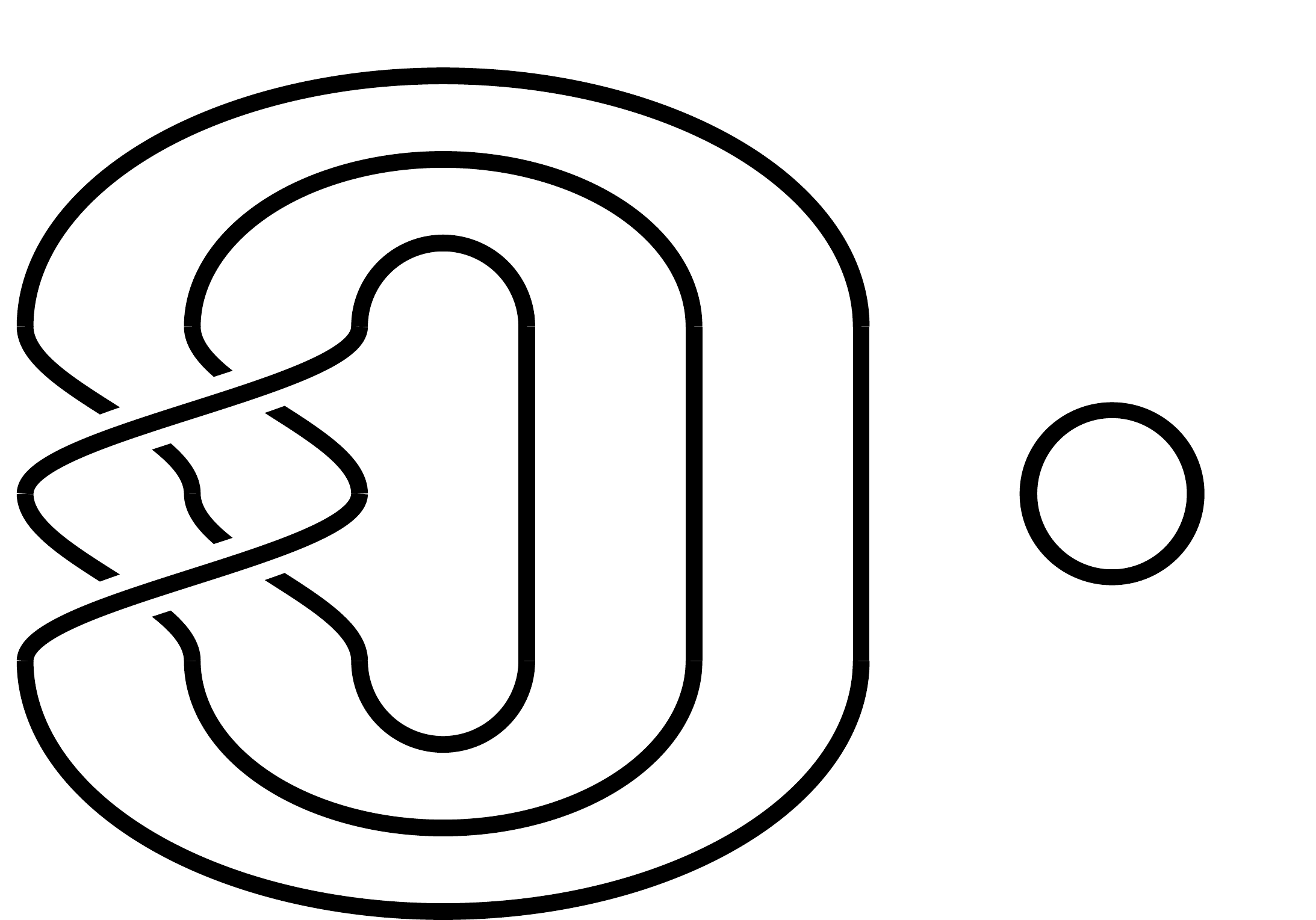{(d)}}
~~~
\subfigure{\label{fig8e}
\def\svgwidth{5cm}
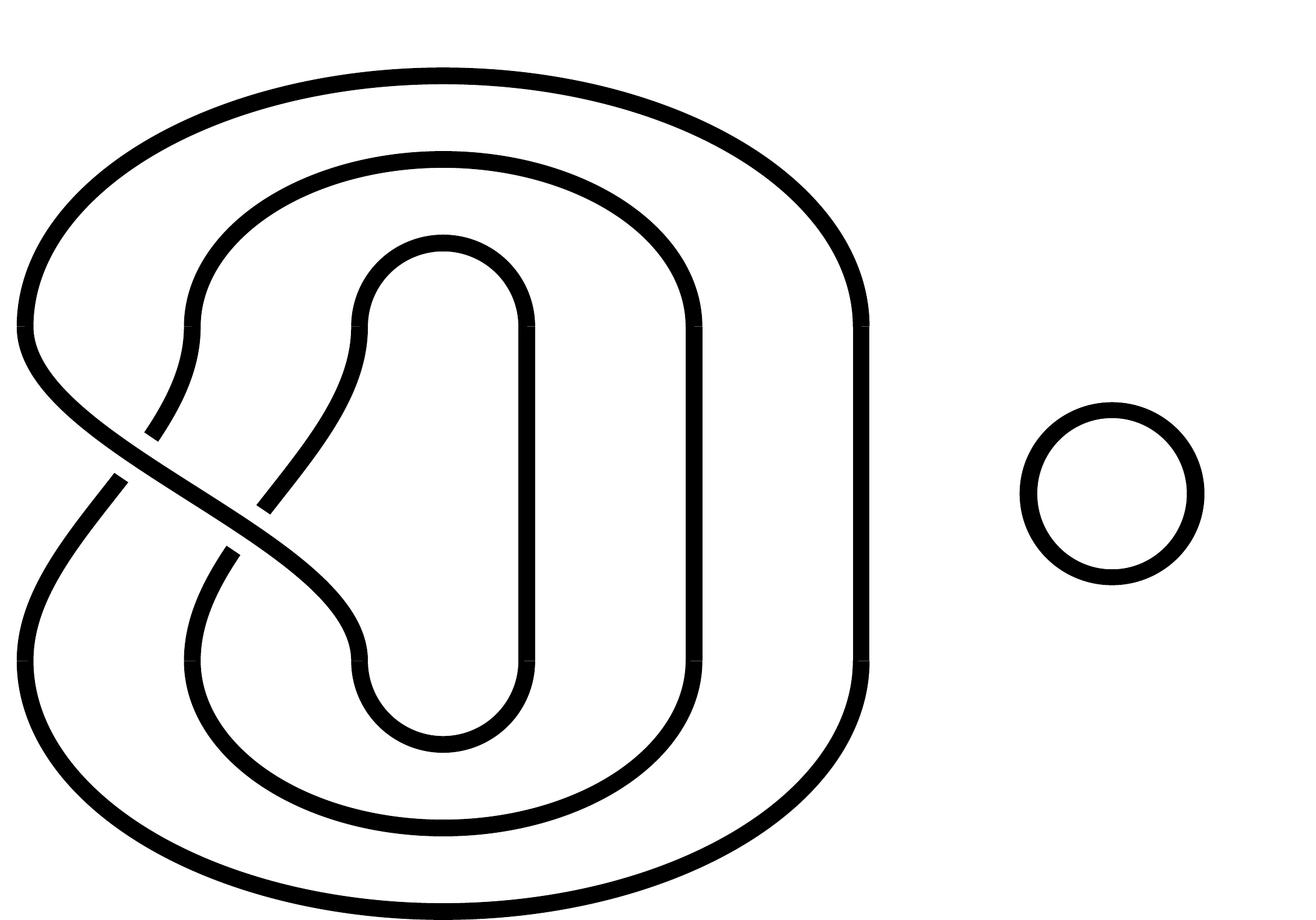{(e)}}
\caption{A sequence of blow up and blow downs showing that $S^3_0(4_1)$ bounds a spin manifold with $b_2=11$ and $\sigma=8$. The characteristic link at each stage is specified by darker curves.}

\end{figure}

This example shows that Theorem \ref{thm:slicing} may obstruct sliceness of $K$ but not of $K\#K$. The following result describes how the obstruction behaves with respect to connected sums. For any knot $K$, let $\mathfrak{s}_1$ denote the spin structure on $S^3_0(K)$ which does not extend over the $4$-manifold produced by attaching a $0$-framed $2$-handle to $D^4$ along $K$.

\begin{prop}\label{prop:sums}

Let $K_1, K_2$ be knots and $X_i$ be a smooth spin 2-handlebody with boundary $(S^3_0(K_i), \mathfrak{s}_1)$ for $i=1,2$.

There is a smooth spin 2-handlebody $X$ with $\partial X = (S^3_0(K_1\# K_2), \mathfrak{s}_1)$, $\sigma(X) = \sigma(X_1)+\sigma(X_2)$ and $b_2(X) = b_2(X_1) + b_2(X_2) + 1$.

\end{prop}
%[Restricting to 2-handlebody rather than general 4-manifold here is easier and adequate for our purposes.]

\begin{proof}
Let $W$ be the 2-handle cobordism from $Y=S^3_0(K_1) \# S^3_0(K_2)$ to $S^3_0(K_1\#K_2)$ illustrated in Figure \ref{nicecobordism}.
Let $X$ be the manifold constructed by attaching $W$ to $X_1 \natural X_2$ along $Y$.

\begin{figure}[htbp] 
\begin{center}
\small
\def\svgwidth{12cm}
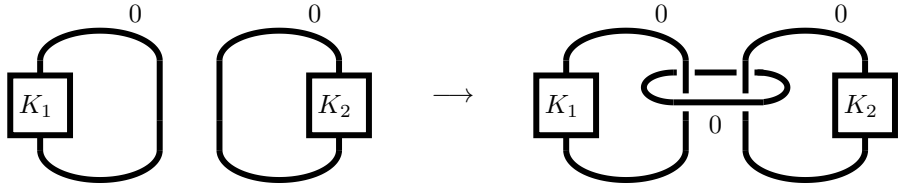
\caption{
{$2$-handle cobordism $W:S^3_0(K_1) \# S^3_0(K_2) \to S^3_0(K_1\#K_2)$.}}
\label{nicecobordism}
\end{center}
\end{figure}

The characteristic link for the spin structure $\mathfrak{s}_1$ in $Y$ is the knot $K_1 \# K_2$ and, since the new 2-handle has linking zero with this component, there is a spin structure on $W$ which restricts to $\mathfrak{s}_1 \# \mathfrak{s}_1$ on $Y$ and $\mathfrak{s}_1$ on $S^3_0(K_1\#K_2)$. Consequently, $X$ extends the correct spin structure on its boundary.

It is easy to see that $\sigma(W) = 0$ and so $\sigma(X) = \sigma(X_1) + \sigma(X_2)$. 
Since $X_1, X_2$ and $X$ are all 2-handlebodies
\[b_2(X) = \chi(X) -1 = \chi(X_1 \natural X_2) +\chi(W) -1 = 1+b_2(X_1) + b_2(X_2).\]

%Let $T$ be the 2-handlebody formed by attaching 0-framed 2-handles to $D^4$ along the split link with components $K_1$ and $K_2$ and $T'=T \cup W$. The exact sequence for the pair $(W,Y)$ and excision isomorphisms between the pairs $(T',T)$ and $(W,Y)$ show that $b_2(W)-b_1(W)=1$. Then the Mayer-Vietoris sequence applied to $X=(X_1 \natural X_2) \cup_Y W$ confirms that $b_2(X) = b_2(X_1) + b_2(X_2) + 1$.
%using 2-handlebodies to clean out details.
\end{proof}

%\begin{rmk}
%Theorem \ref{thm:slicing} cannot be used to show that any knot has infinite concordance order.
%\end{rmk}
\begin{rmk}
The signature of any spin manifold with spin boundary $(S^3_0(K),\mathfrak{s}_1)$ is $8 \operatorname{Arf} K \mod 16$, where $\operatorname{Arf} K$ is the Arf invariant of the knot $K$. (See \cite{Saveliev2002}.) Note that after removing the characteristic link, to get to a spin manifold bounded by the $0$-surgery on $K$, the signature must be a multiple of $8$.
\end{rmk}
%\textcolor{red}{I would say we move the other remark to the last section.}
\subsection{A topologically slice example}

Let $K$ be the knot shown in Figure \ref{fig:k6}. A straightforward calculation of the Alexander polynomial shows that $\Delta_K(t)=1$ and so $K$ is topologically slice. See~\cite[11.7B~Theorem]{Freedman1990}, \cite[Theorem~7]{Freedman1984}. See also~\cite[Appendix~A]{Garoufalidis2004} and \cite{Cha2014}.

\begin{figure}[htbp] 
\begin{center}
\small
\def\svgwidth{10cm}
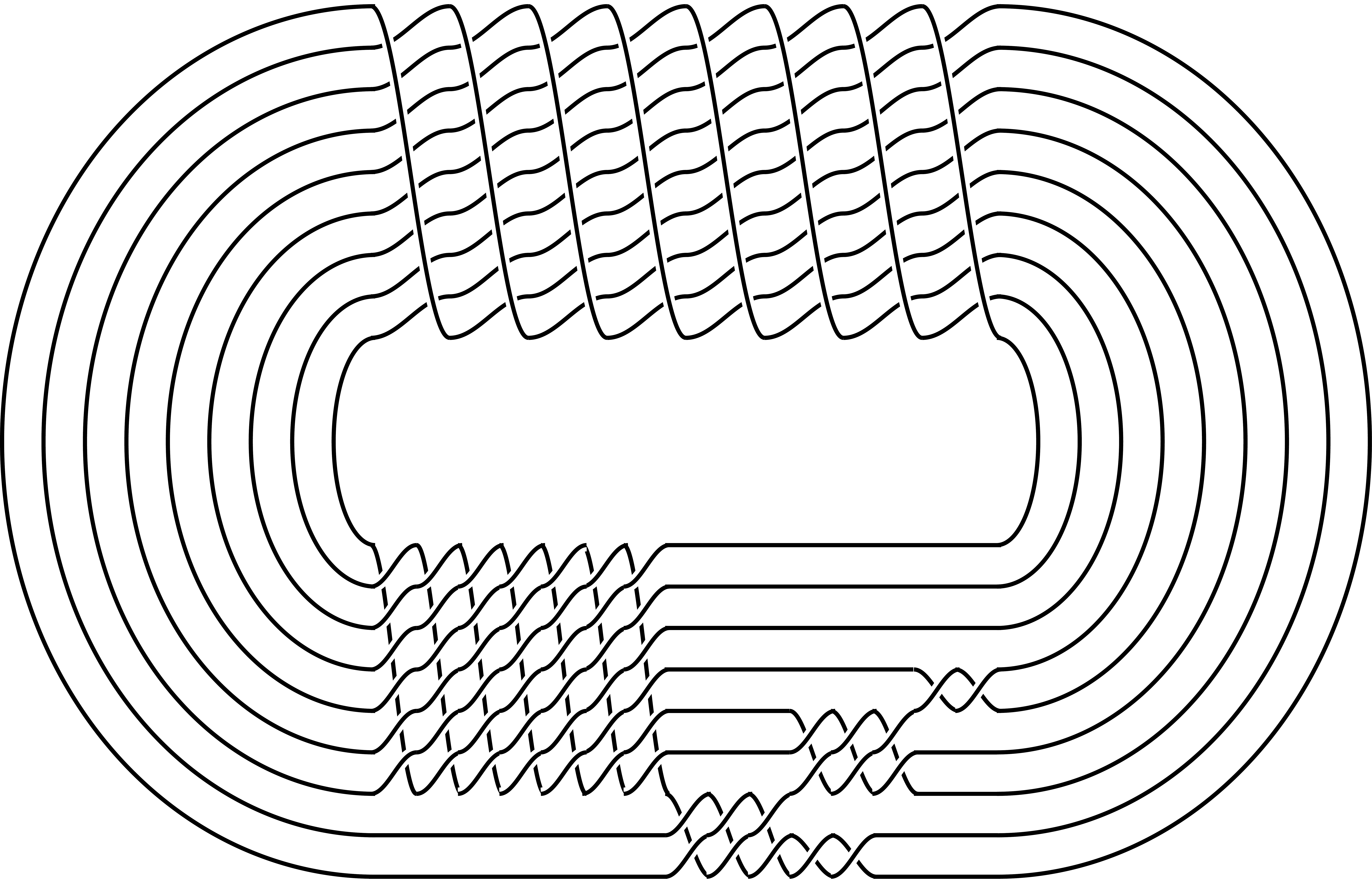
\caption{
{A topologically slice knot obtained as the closure of the braid word $(\sigma_8 \sigma_7 ... \sigma_1)^8 (\sigma_{3}\sigma_{4} ... \sigma_{8})^{-7} (\sigma_{1}\sigma_{2})^{-3} \sigma_{1}^{-2}(\sigma_{3}\sigma_{4})^{-3}\sigma_{5}^{-2}$.}}
\label{fig:k6}
\end{center}
\end{figure}

\begin{examp}
$K$ is not smoothly slice.
\end{examp}

Add a $0$-framed $2$-handle to $\partial D^4$ along $K$ and then blow up three times around the sets of strands indicated in Figure \ref{fig:k6BU}. Blow up negatively across nine strands on the top and positively across five and seven strands on the bottom of the diagram. This gives a manifold with signature $1$ and second Betti number $4$. The characteristic link has one component, as shown in Figure \ref{fig:k6BU2}, with framing $-7$. An isotopy verifies that this knot is $4_1$.

\begin{figure}[htbp] 
\begin{center}
\small
\def\svgwidth{10cm}
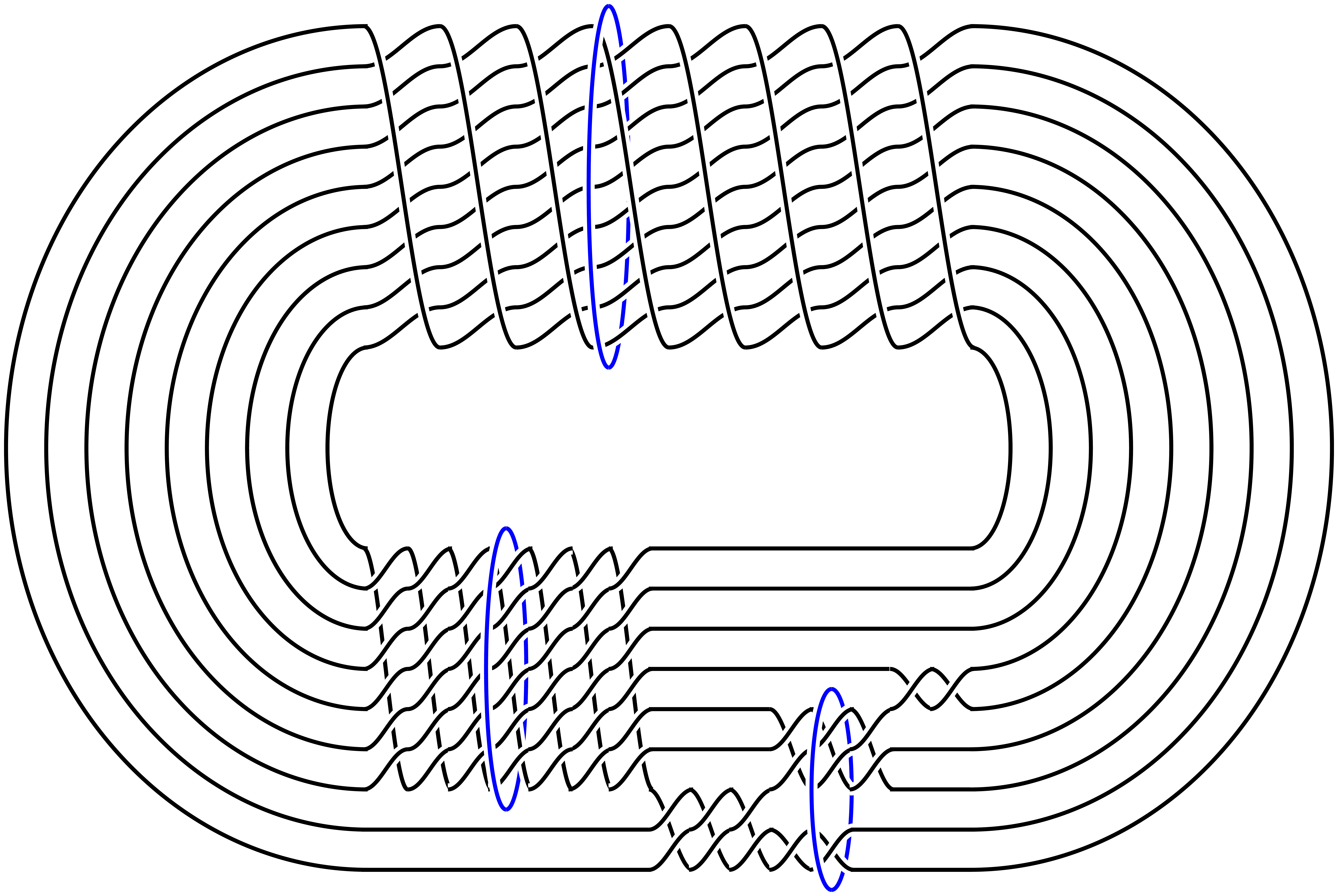
\caption{
{$K$ can be simplified by blowing up along the blue curves with appropriate signs. Note that none of the blue curves will be part of the characteristic link.}}
\label{fig:k6BU}
\end{center}
\end{figure}

\begin{figure}[htbp] 
\begin{center}
\small
\def\svgwidth{8cm}
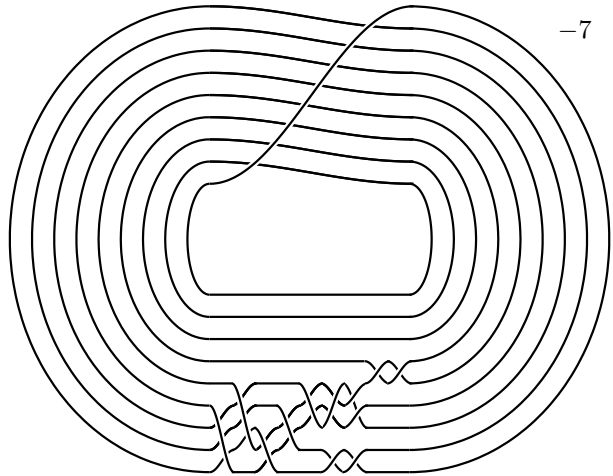
\caption{
{Characteristic link is a $-7$-framed figure-eight. }}
\label{fig:k6BU2}
\end{center}
\end{figure}

Following the procedure from Example \ref{eg:figureeight}, we may blow up negatively three times to produce a characteristic link which is a two-component unlink with framings $-2$ and $-16$ in a manifold with $\sigma=-2$ and $b_2=7$. Blow up meridional curves of this unlink until the framing coefficients are both $-1$, then blow down the resulting $-1$-framed unlink. This yields a spin manifold with signature $16$ and second Betti number $21$. Therefore, by Theorem~\ref{thm:slicing}, $K$ is non-slice.

Note that Figure \ref{fig:k6} presents $K$ as a generalized twisted torus knot. It is the closure of a braid formed by taking a $(9,8)$ torus knot and then adding negative full twists on seven strands, then on non-adjacent sets of three strands and finally a pair of negative clasps. The obstruction from Theorem \ref{thm:slicing} is generally easier to apply to knots like this because they can be unknotted efficiently by blowing up to remove full twists. For many twisted torus knots this provides a slicing obstruction which is often more easily computable than the signature function.

It would be interesting to find other examples where this obstruction applies. It may be able to obstruct smooth sliceness for Whitehead doubles. To apply Theorem~\ref{thm:slicing}, we need the sequence of blow-up moves to predominantly involve blow-ups of the same sign.  However, at least for the standard diagrams of Whitehead doubles, it is not easy to see how to do this. Similarly, it should be possible to detect other torsion elements of the knot concordance group. Example~\ref{eg:figureeight} demonstrates this in principle but it would be interesting to obtain new examples of torsion elements.% Note that Proposition~\ref{prop:sums} implies that this obstruction is unable to show that any knot has infinite order. %%This is not adequately justified or important enough to try harder on.

%\textcolor{red}{We should have another section explaining the possible use of the obstruction for detecting more knots of order two in the concordance group. It is good to mention at this point that our obstruction is not capable to detect knots with infinite order in the concordance group. I guess it is worth mentioning that the easiest way to detect the non-sliceness of Whitehead doubles is to use the $\tau$ invariant, though $\tau$ vanishes for the figure eight. So Whitehead doubles are a possible family of examples that this obstruction might work for.}

%\clearpage

%\begin{figure}[htbp] 
%\begin{center}
%\small
%\def\svgwidth{6cm}
%\input{k6afterBU2.pdf_tex}
%\caption{
%{Characteristic link is $-7$ framed figure-eight. Bit of simple isotopy. Is it needed?}}
%\label{fig:k6BU3}
%\end{center}
%\end{figure}
%\begin{figure}[htbp] 
%\begin{center}
%\small
%\def\svgwidth{5cm}
%\input{k6afterBU3.pdf_tex}
%\caption{
%{Characteristic link is $-7$ framed figure-eight. Bit more isotopy and it should be easily identifiable now.}}
%\label{fig:k6BU4}
%\end{center}
%\end{figure}

\bibliographystyle{amsalpha2}

\bibliography{Reference}

\end{document}